\documentclass[12pt]{amsart}
\usepackage{fullpage}
\usepackage{times}
\usepackage{latexsym}
\usepackage{amssymb}
\usepackage[all]{xy}

\newtheorem{thm}{Theorem}[section]
\newtheorem{prop}[thm]{Proposition}
\newtheorem{cor}[thm]{Corollary}

\newtheorem{que}[thm]{Question}

\theoremstyle{remark}
\newtheorem{rem}[thm]{Remark}
\newtheorem{defn}{Definition}[section]
\newtheorem{ex}[thm]{Example}

\usepackage{hyperref}

\newcommand{\Q}{\mathbb{Q}}
\newcommand{\R}{\mathbb{R}}
\newcommand{\Z}{\mathbb{Z}}

\newcommand{\Out}{\mathrm{Out}}
\newcommand{\Aut}{\mathrm{Aut}}
\newcommand{\Sym}{\mathrm{Sym}}

\newcommand\id{\operatorname{id}}

\newcommand\im{\operatorname{im}}

\title[Anosov diffeomorphisms of products I]
      {Anosov diffeomorphisms of products I. \\Negative curvature and rational homology spheres}

\author{Christoforos Neofytidis}
\address{Section de Math\'ematiques, Universit\'e de Gen\`eve, 2-4 rue du Li\`evre, Case postale 64, 1211 Gen\`eve 4, Switzerland}
\email{Christoforos.Neofytidis@unige.ch}
\date{\today}

\begin{document}

\begin{abstract}
We show that various classes of products of manifolds do not support transitive Anosov diffeomorphisms. Exploiting the Ruelle-Sullivan cohomology class, we prove that the product of a negatively curved manifold with a rational homology sphere does not 
support transitive Anosov diffeomorphisms. We extend this result to products of finitely many negatively curved manifolds of dimensions at least three with a rational homology sphere that has vanishing simplicial volume. As an application of this study, we obtain new examples of manifolds  
that do not support transitive Anosov diffeomorphisms, including certain manifolds with non-trivial higher homotopy groups and certain products of aspherical manifolds. 
\end{abstract}

\maketitle

\section{Introduction}

A diffeomorphism $f$ on a closed oriented smooth $n$-dimensional manifold $M$ is called {\em Anosov} if there exist constants $\mu\in (0,1)$ and $C > 0$, together with a $df$-invariant splitting $TM=E^s\oplus E^u$ of the tangent bundle of $M$, such that for all $m\geq 0$
\begin{equation*}
\begin{split}
\|df^m(v)\| \leq C\mu^m\|v\|,\ v \in E^s,\\  
\|df^{-m}(v)\| \leq C\mu^m\|v\|,\ v \in E^u.
\end{split}
\end{equation*}

The invariant distributions $E^s$ and $E^u$ are called the {\em stable} and {\em unstable} distributions. An Anosov diffeomorphism $f$ is called a {\em codimension $k$ Anosov diffeomorphism} if either fiber $E^s$ or $E^u$ has dimension $k$ with $k \leq [n/2]$, and is called {\em transitive} if there exists a point whose orbit is dense in $M$.

All examples of Anosov diffeomorphisms known to date are conjugate to affine automorphisms of infranilmanifolds, and it is a long-standing question, going back to Anosov and Smale, whether there exist any other manifolds that support Anosov diffeomorphisms. In particular, Smale suggested that every manifold which supports an Anosov diffeomorphism must be covered by a Euclidean space~\cite{Smale}. 
Obstructions to the existence of (transitive) Anosov diffeomorphisms have been developed mainly using co-homological and coarse geometric methods. In the following theorem we gather  
some of those obstructions 
and some major examples of manifolds that do not support (transitive) Anosov diffeomorphisms: 

\begin{thm}
\label{results} \
\begin{itemize}
\item[(a)] {\normalfont(Franks \cite{Fr}, Newhouse \cite{N}).} If a manifold supports a codimension one Anosov diffeomorphism, then it is homeomorphic to a torus.
\item[(b)] {\normalfont(Ruelle-Sullivan \cite{RS}).} If $M$ supports a codimension $k$ transitive Anosov diffeomorphism $f$ with orientable invariant distributions, then there is a non-trivial cohomology class $\alpha\in H^k(M;\R)$ and a positive $\lambda\neq 1$ such that $f^*(\alpha) = \lambda\cdot\alpha$. In particular,  $H^k(M;\R)\neq 0$.
\item[(c)] {\normalfont(Shiraiwa \cite{Sh}).} If $f\colon M\longrightarrow M$ is an Anosov diffeomorphism, then there is some $k$ such that the induced homomorphism $f^*\colon H^k(M;\Q)\longrightarrow H^k(M;\Q)$ is not the identity. In particular, rational homology spheres do not support Anosov diffeomorphisms.
\item[(d)] {\normalfont(Yano \cite{Yano}).} A manifold with negative sectional curvature does not support transitive Anosov diffeomorphisms.
\end{itemize}
\end{thm}

In the sequel, we will refer to the cohomology class of Theorem \ref{results} (b) as {\em Ruelle-Sullivan cohomology class}. In Section \ref{simplicialnorm}, we will see how this class can be used to rule out transitive Anosov diffeomorphisms on negatively curved manifolds, which is Yano's Theorem \ref{results} (d). We note that in dimensions higher than two the transitivity assumption in Yano's result  is in fact not needed as explained in \cite[Cor. 4.5]{GL}; we will discuss this briefly in Section \ref{outer}. 

In recent years, there have been several attempts to extend the above non-existence results to Cartesian products:

\begin{que}\label{que1}
\label{p:anosovdiffeo}
Let $M$ and $N$ be two closed manifolds such that at least one of them does not support Anosov diffeomorphisms. Does $M\times N$ support an Anosov diffeomorphism?
\end{que}

Concerning, in one direction, manifolds with non-trivial higher homotopy groups, and, in another direction, aspherical manifolds, the following results were proven recently:

\begin{thm}\label{resultsGHGL}\
\begin{itemize}
\item[(a)] {\normalfont(Gogolev-Rodriguez Hertz \cite{GH}).} Let $M$ be a closed $m$-dimensional manifold. If $n > m$, then the product $M \times S^n$ does not support transitive Anosov diffeomorphisms. Moreover, if $n$ is odd, then $M \times S^n$ does not support Anosov diffeomorphisms.
\item[(b)] {\normalfont(Gogolev-Lafont \cite{GL}).} Let $N$ be a closed infranilmanifold and $M$ be a closed smooth aspherical manifold such that
 $\pi_1(M)$ is Hopfian, 
 the outer automorphism group $\Out(\pi_1(M))$ is finite, and the intersection of all maximal nilpotent subgroups of $\pi_1(M)$ is trivial. 
Then $M \times N$ does not support Anosov diffeomorphisms.
\end{itemize}
\end{thm}

Our goal is to provide further evidence towards a positive answer to the Anosov-Smale question, by extending the classes of product manifolds that do not support (transitive) Anosov diffeomorphisms. 
In this paper, we consider products of negatively curved manifolds with a rational homology sphere. In particular, we obtain new examples of manifolds that do not support transitive Anosov diffeomorphisms, including certain manifolds with non-trivial higher homotopy groups and certain aspherical manifolds. 
Our first result concerns products of a negatively curved manifold with a rational homology sphere:

\begin{thm}\label{t:main1}
If $M$ is a closed negatively curved manifold and $N$ is a rational homology sphere, 
then the product $M\times N$
does not support transitive Anosov diffeomorphisms with orientable invariant distributions.
\end{thm}

In particular, we obtain the following immediate consequence that includes manifolds with non-trivial higher homotopy groups; compare Theorem \ref{resultsGHGL} (a):

\begin{cor}\label{c:sphere}
For any negatively curved manifold $M$, the product $M\times S^n$ does not support transitive Anosov diffeomorphisms.
\end{cor}

Ruelle-Sullivan~\cite{RS} suggested that the study of Anosov diffeomorphisms on a manifold can hopefully be reduced to the study of the algebraic topology of that manifold. 
In order to prove Theorem \ref{t:main1}, we will study the cohomology ring of $M\times N$ and show that a Ruelle-Sullivan class cannot exist for those products. To this end, we will use several properties that are implied by the presence of the negatively curved factor $M$. Among those properties is the fact that every non-trivial homology class $\alpha\in H_k(M;\R)$ has non-zero simplicial semi-norm when $k>1$~\cite{G,IY}. Also, we will show that such a class $\alpha$ cannot be realized by products of homology classes of lower degrees; cf. Corollary \ref{not representable}. Furthermore, an important tool is that the outer automorphism group $\Out(\pi_1(M))$ is finite when $\dim M \geq 3$; see Section \ref{outer}. 

Theorem \ref{t:main1} should still hold if we replace $M$ with a finite product $M_1\times\cdots\times M_s$ of negatively curved manifolds. 
Passing however to $M_1\times\cdots\times M_s$, $s\geq2$, 
the positivity of the simplicial semi-norm of non-trivial homology classes of degrees greater than one does not hold anymore~\cite{G,IY}, and, of course, homology classes in $H_k(M_1\times\cdots\times M_s;\R)$ can be products themselves. Nevertheless, every homology class of each factor $M_i$ is still not realizable by products and has non-zero simplicial semi-norm (in degrees $>1$). Moreover, when $\dim (M_i)\geq 3$ for all $i=1,...,s$, then $\Out(\pi_1(M_1\times\cdots\times M_s))$ remains finite; cf. Section \ref{outer}. With the additional assumption that the rational homology sphere has vanishing simplicial volume we obtain the following:

\begin{thm}\label{t:main2}
Let $M_1,...,M_s$ be closed negatively curved manifolds of dimensions greater than two and $N$ be a rational homology sphere which has vanishing simplicial volume. Then the product $M_1 \times\cdots\times M_s \times N$ does not support transitive Anosov diffeomorphisms with orientable invariant distributions.
\end{thm}

Since spheres have vanishing simplicial volume we obtain:

\begin{cor}\label{c:sphere1}
Let $M_1,...,M_s$ be closed negatively curved manifolds of dimensions greater than two. Then the product $M_1 \times\cdots\times M_s \times S^n$ does not support transitive Anosov diffeomorphisms.
\end{cor}

\begin{rem}\label{rem1}
Note that removing the assumption on the dimensions of the $M_i$ in Theorem \ref{t:main2} seems to be more subtle than removing the assumption on the simplicial volume of $N$. Namely, it is still an open question whether the product of two hyperbolic surfaces or, more generally, of a hyperbolic surface with a higher dimensional negatively curved manifold supports (transitive) Anosov diffeomorphisms; cf.~\cite[Section 7.2]{GL}. Using Theorem \ref{t:main1} we will be able to obtain examples of products of negatively curved manifolds that do not support transitive Anosov diffeomorphisms, where one of the factors can be a hyperbolic surface; see Section \ref{examples}.
\end{rem}

\subsection*{Outline of the paper} 
In the next section we overview briefly some basic obstructions to the existence of (transitive) Anosov diffeomorphisms. In Section \ref{realization}, we discuss realizability of homology classes of negatively curved manifolds by products of homology classes. 
In Sections \ref{proof1} and \ref{proof2} we prove Theorems \ref{t:main1} and \ref{t:main2} respectively and obtain examples of manifolds that do not support transitive Anosov diffeomorphisms.

\subsection*{Acknowledgments}
I am grateful to Andrey Gogolev for several fruitful discussions. Also, I would like to thank Dennis Sullivan for his useful comments.

\section{Simplicial semi-norm and outer automorphism groups}

In this preliminary section we recall some basic obstructions to the existence of (transitive) Anosov diffeomorphisms that apply in particular to negatively curved manifolds and partially to their products.

\subsection{Simplicial semi-norm}\label{simplicialnorm}

Let $X$ be a topological space and $\alpha\in H_k(X;\R)$. The simplicial semi-norm (or Gromov's norm) of $\alpha$ is defined by
\[
 \|\alpha\|_1 := \inf_c \biggl\{\sum_j |\lambda_j| \ \biggl| \ c = \sum_j \lambda_j\sigma_j \in C_k(X;\R) \ \text{is a cycle representing } \alpha\biggl\}.
\]
If $M$ is a closed oriented manifold and $[M]\in H_{\dim M}(M)$ denotes its fundamental class, then $\|M\| := \|[M]\|_1$ is called {\em the simplicial volume of $M$}. 

An important property of the simplicial semi-norm is that it is functorial, i.e. $f_*\colon H_*(X)\longrightarrow H_*(Y)$ is not increasing under any continuous map $f\colon X\longrightarrow Y$. Especially, if $f\colon M\longrightarrow N$ is a map of degree $d$, then $\|M\|\geq |d|\cdot\|N\|$. This means that manifolds that admit a self-map of degree greater than one have vanishing simplicial volume. For instance, since for each $n$ the sphere $S^n$ has infinite set of self-mapping degrees (in fact, equal to $\Z$), we deduce that $\|S^n\|=0$.

A prominent result of Gromov~\cite{G} and Inoue-Yano~\cite{IY} states that the simplicial semi-norm of homology classes of a negatively curved manifold $M$ is positive in all degrees greater than one. Using this and the Ruelle-Sullivan Theorem \ref{results} (b), Yano~\cite{Yano} proved that negatively curved manifolds do not support transitive Anosov diffeomorphisms (Theorem \ref{results} (d)): First, by Theorem \ref{results} (a), we may assume that $\dim M \geq 3$ and that the codimension $k$ of each Anosov diffeomorphism $f\colon M\longrightarrow M$ is greater than one. Now, if $k$ is the dimension of the stable bundle $E^s$, then 
$\lambda=e^{-h_{top}(f)}<1$
 in Theorem \ref{results} (b), where $h_{top}(f)$ is the topological entropy of $f$; cf.~\cite[pg. 326]{RS}. Thus, the Poincar\'e dual $\beta\in H_{\dim M-k}(M;\R)$ of the Ruelle-Sullivan class $\alpha$ satisfies $f_*(\beta)=\pm\frac{1}{\lambda}\cdot \beta$. Since $\frac{1}{\lambda}>1$, we deduce that $\|\beta\|_1=0$, which is impossible because $M$ is negatively curved and $\dim M-k>1$ (since $\dim M\geq 3$ and $k\leq[\frac{\dim M}{2}]$). Thus $k$ cannot be the dimension of $E^s$ and so it must be the dimension of the unstable bundle $E^u$. In that case, $\lambda=e^{h_{top}(f)}>1$ and so the Kronecker dual $b\in H_k(M;\R)$ of $\alpha$ has zero simplicial semi-norm. This contradiction completes the proof.

\subsection{Outer automorphism groups}\label{outer}

As explained\footnote{See the arXiv preprint of~\cite{GL}, version 1511.00261v1, Lemma 2.1.} by Gogolev-Lafont~\cite{GL}, aspherical manifolds whose fundamental group has torsion outer automorphism group do not support transitive Anosov diffeomorphisms. For let $M$ be an aspherical manifold with torsion $\Out(\pi_1(M))$ and suppose that there exists an Anosov diffeomorphism $f\colon M\longrightarrow M$. Since $\Out(\pi_1(M))$ is torsion, there is some $l$ such that $(\pi_1(f))^l$ is an inner automorphism of $\pi_1(M)$, and so $((\pi_1(f))^l)^*\colon H^*(\pi_1(M);\R)\longrightarrow H^*(\pi_1(M);\R)$ is the identity. Since $M$ is aspherical, the commutative diagram in Figure \ref{f:outer} implies that $(f^l)^*$ must be the identity on $H^*(M;\R)$. This is however impossible by Ruelle-Sullivan's Theorem \ref{results} (b) or by Shiraiwa's Theorem~\ref{results} (c) (note that, after passing to finite coverings if necessary, we may assume that $f$ has orientable invariant distributions).

\begin{figure}
     \[
\xymatrix{
& H^*(M;\R)
\ar[dd]_{\simeq} \ar[rr]^{f^*} & &  H^*(M;\R) \ar[dd]^{\simeq}\\
& & & \\
& H^*(\pi_1(M);\R) \ar[rr]^{(\pi_1(f))^*} & & H^*(\pi_1(M);\R)\\
}
     \]
\caption{\small{Anosov diffeomorphisms of aspherical manifolds on the level of group cohomology.}}
\label{f:outer}
 \end{figure}   
 
 In fact, the above result holds without the transitivity assumption (using the Lefschetz number), as pointed out for instance by Gogolev and Lafont~\cite[Lemma 4.1]{GL}. Appealing furthermore to results of Paulin~\cite{Paulin}, Bestvina-Feighn~\cite{BF} and Bowditch~\cite{Bow}, Gogolev-Lafont then deduced that negatively curved manifolds of dimension greater than two do not support Anosov diffeomorphisms, waving thus the transitivity assumption in Yano's Theorem~\ref{results} (d). 
     An important observation in~\cite{GL} is that this result can be extended to direct products of finitely many negatively curved manifolds of dimensions greater than two. For let $\Gamma$ be the fundamental group of a product of negatively curved manifolds $M_1\times\cdots\times M_s$ and denote by $\Gamma_i$ the fundamental group of each factor $M_i$, $i=1,...,s$. Since each $\Gamma_i$ is hyperbolic and thus does not contain any subgroup isomorphic to $\Z^2$, every automorphism $\varphi\colon\Gamma\longrightarrow\Gamma$ permutes those $\Gamma_i$. In particular, there is a homomorphism from the automorphism group $\Aut(\Gamma)$ to the symmetric group $\Sym(s)$, which clearly gives rise to a homomorphism $\Out(\Gamma)\longrightarrow\Sym(s)$. The kernel of the latter homomorphism is $\Out(\Gamma_1)\times\cdots\times\Out(\Gamma_s)$, which implies that $\Out(\Gamma)$ is finite since each $\Out(\Gamma_i)$ is finite (because $\Gamma_i$ is hyperbolic; cf.~\cite{Paulin,BF}). 

Using the finiteness of the outer automorphism group in showing non-existence of Anosov diffeomorphisms on direct products of negatively curved manifolds of dimensions at least three is particularly important, because the simplicial semi-norm of a homology class in $H_k(M_1\times\cdots\times M_s;\R)$ might be zero when 
     \[
     k<\sum_{i=1}^s \dim M_i -\min_{1\leq i\leq s}\{\dim M_i\} +2
     \]
 (see~\cite{IY,G}), and thus cannot contradict the existence of a Ruelle-Sullivan class.

\section{Realization of homology classes by products}\label{realization}

In this section we discuss briefly Thom's work on realizing co-homology classes by closed oriented manifolds, and obtain a result about the non-realizability of co-homology classes of negatively curved manifolds by products of co-homology classes of lower dimensions.

\subsection{Thom's Realization Theorem}\label{Thom}

Given a topological space $X$, Steenrod~\cite[Problem 25]{Eilenberg:Steenrod} raised the question of whether every integral homology class $\alpha\in H_k(X;\Z)$ can be realized by a manifold, i.e. whether there is a closed oriented $k$-dimensional manifold $M$ together with a continuous map $f\colon M\longrightarrow X$ such that $f_*([M])=\alpha$.
Thom \cite{Thom} answered affirmatively Steenrod's question in degrees up to six and in any
degree in homology with $\Z_2$ coefficients. However, he showed that there exists a $7$-dimensional integral homology class which is not
realizable by a manifold (since then, other non-realizability results have been obtained). Nevertheless, Thom proved that in all degrees some multiple of every integral
homology class is realizable by a closed oriented smooth manifold. 
In particular, every real homology class in any degree $k$ is realizable by a closed oriented smooth $k$-dimensional manifold $M$.

Thom's Realization Theorem~\cite{Thom} reads as follows in cohomology: 

\begin{thm}[Thom's Realization Theorem in cohomology]\label{t:Thom}
For every $\alpha\in H^k (X;\Z)$, there exists an integer $d > 0$ and a closed oriented smooth $k$-dimensional manifold $M$ together with a continuous map $f\colon M \longrightarrow X$ so that $H^k(f)(\alpha)=d\cdot\omega_M$, where $\omega_M \in H^k(M)$ denotes the cohomological fundamental class of $M$. 
\end{thm}

In the proof of Theorems \ref{t:main1} and \ref{t:main2} we will use this form of Thom's theorem.

\subsection{Homology classes not realizable by products}\label{realizableproducts}

A special case of Steenrod's realizability question is whether a given homology class $\alpha\in H_k(X;\R)$ can be realized by a non-trivial direct product of manifolds. In the light of Thom's theorem, this question is equivalent to asking whether $\alpha$ is realizable by a product of homology classes.

\begin{defn}
Let $X$ be a topological space. A homology class $\alpha\in H_k(X;\R)$ is said to be {\em realizable (or representable) by products} if there exist spaces $X_1, X_2$ together with a continuous map $f\colon X_1\times X_2\longrightarrow X$ such that $f_*(\alpha_i\times\alpha_{k-i})=\alpha$ for some $\alpha_i\in H_i(X_1;\R)$ and $\alpha_{k-i}\in H_{k-i}(X_2;\R)$. 
\end{defn}

\begin{rem}
Similarly, we say that a cohomology class $x\in H^k(X;\R)$ {\em is realizable by products} if there exist spaces $X_1, X_2$ together with a continuous map $f\colon X_1\times X_2\longrightarrow X$ such that $f^*(x)=x_i\times x_{k-i}$ for some $x_i\in H^i(X_1;\R)$ and $x_{k-i}\in H^{k-i}(X_2;\R)$.
\end{rem}

When $\alpha$ is an integral homology class, then we ask about realizability of $\alpha$ up to multiples, i.e. whether there are spaces $X_1, X_2$ together with a continuous map $f\colon X_1\times X_2\longrightarrow M$ such that $f_*(\alpha_i\times\alpha_{k-i})=d\cdot\alpha$, where $d$ is a non-zero integer. In particular, when $X=M$ is a closed oriented $k$-dimensional manifold and $\alpha=[M]$, then the question is whether $M$ is {\em dominated by products}, i.e. whether $M$ admits a map of non-zero degree from a non-trivial direct product.

Domination by products has been studied extensively during the last decade. One of the initial attempts to find obstructions to the existence of such maps for large classes of manifolds was made by Kotschick and L\"oh~\cite{KL}, who introduced a group theoretic property 
for {\em essential} manifolds, i.e. closed oriented manifolds $M$ satisfying 
\[
H_{\dim M}(c_M)([M])\neq 0 \in H_{\dim M}(B\pi_1(M)), 
\]
where
$c_M \colon M \longrightarrow B\pi_1(M)$ classifies the universal covering of $M$. 
That group theoretic property reads as follows: 

\begin{defn}
An infinite group $\Gamma$ is called {\em not presentable by products} if, for every homomorphism $\varphi \colon \Gamma_1 \times \Gamma_2 \longrightarrow
\Gamma$ onto a finite index subgroup of $\Gamma$, the restriction of $\varphi$ to one of the factors $\Gamma_i$ has finite image $\varphi(\Gamma_i)$. 
\end{defn}

The main result of~\cite{KL} is that any rationally (equiv. really) essential manifold with fundamental group not presentable by products cannot be dominated by products. Prominent examples of (rationally) essential manifolds whose fundamental group is not presentable by product are non-positively curved manifolds of dimension at least two that are not covered by products.

Given a rationally essential manifold $M$ with fundamental group not presentable by products, it is natural to ask whether homology classes in degrees less than $\dim M$ are realizable by products. In general, the answer is negative as one can see in the following example:

\begin{ex}\label{realizablenonpresentable}
Let $M$ be the mapping torus of a hyperbolic automorphism of the $2$-torus $T^2$. Then $M$ is a closed aspherical $3$-manifold (modelled on the geometry $Sol^3$) and thus (rationally) essential. Also, $\pi_1(M)$ is not presentable by products; see for example~\cite{KN}. Thus $M$ is not dominated by products. However the image of the fiber $T^2$ of $M$ represents a non-trivial product homology class in $H_2(M)$. 
\end{ex}

As in the situation of the above example, observe that a continuous map $f\colon X_1\times X_2\longrightarrow M$ with $f_*([X_1\times X_2])=\alpha\in H_k(M;\R)$, where $\dim X_1+\dim X_2<\dim M$, is in general far away from being $\pi_1$-surjective, and thus cannot induce a presentation by products for $\pi_1(M)$ (in Example \ref{realizablenonpresentable} we have $X_1=X_2=S^1$ and $f$ is the inclusion). 
Nevertheless, such a map still produces two commuting subgroups of $\pi_1(M)$. For aspherical manifolds with hyperbolic fundamental group this alone provides an obstruction to realizability by products in any degree:

\begin{thm}\label{presentability}
Suppose $M$ is a closed oriented aspherical manifold with hyperbolic fundamental group. Then any non-trivial class $\alpha\in H_k(M;\R)$ is not realizable by products.
\end{thm}

The proof follows the argument of~\cite[Theorem 1.4]{KL} with the basic difference that it applies to homology classes of any degree and not just the top degree. Also, the asphericity assumption on $M$ is required to ensure that, for all $k$, the classifying map of the universal covering sends each non-trivial homology class of $H_k(M;\R)$ to a non-trivial element in $H_k(B\pi_1(M);\R)$ (namely, to itself in this case).

\begin{figure}
     \[
\xymatrix{
X_1 \times X_2 
\ar[dd]|-{(B\pi_1(f\vert_{X_1}) \circ c_{X_1}) \times (B\pi_1(f\vert_{X_2}) \circ
c_{X_2})} \ar[rrr]^{f} & & & M \ar[dd]^{c_M\simeq \id_M}\\
& & & \\
B\Gamma_1 \times B\Gamma_2 \ar[rrr]^{B\varphi} & & & B\pi_1(M)\\
}
     \]
\caption{\small{Realizability by products on the level of classifying spaces.}}
\label{f:KotschickLoeh}
 \end{figure}

\begin{proof}[Proof of Theorem \ref{presentability}]
Clearly, we can assume that $\dim M\geq 2$ and so $\pi_1(M)$ is not (virtually) cyclic.
Let $\alpha\in H_k(M;\R)$ be a non-trivial homology class. By Thom's realization theorem (cf. Section \ref{Thom}), it suffices to show that there are no closed oriented manifolds $X_1, X_2$ of positive dimensions with $\dim X_1 +\dim X_2=k$ and a continuous map $f \colon X_1 \times X_2 \longrightarrow M$ such that $f_*[X_1\times X_2]=\alpha$. 

Suppose that such manifolds $X_1, X_2$ exist and let $\pi_1(f) \colon \pi_1(X_1) \times \pi_1(X_2)\longrightarrow \pi_1(M)$ be the $\pi_1$-induced map.
Set $\Gamma_i := \im(\pi_1(f\vert_{X_i})) \subset \pi_1(M)$ for the images under $\pi_1(f)$ of the restrictions of $f$ to the two factors $X_i$. Then the multiplication map
$\varphi \colon\Gamma_1 \times \Gamma_2 \longrightarrow \pi_1(M)$ is a well-defined homomorphism because the $\Gamma_i$ commute elementwise. 

Let now $c_{X_i} \colon X_i \longrightarrow B\pi_1(X_i)$ be the classifying maps of the universal coverings of the $X_i$ and the maps $B\pi_1(f\vert_{X_i}) \colon
B\pi_1(X_i) \longrightarrow B\Gamma_i$ induced by $\pi_1(f\vert_{X_i})$ on the level of classifying spaces. Moreover, let $B\varphi \colon
B\Gamma_1 \times B\Gamma_2 \longrightarrow B\pi_1(M)$ be the map induced by $\varphi$ between the classifying spaces (here we use the fact that $B\Gamma_1 \times B\Gamma_2 $ is homotopy equivalent to $B(\Gamma_1 \times \Gamma_2)$). We then have for $i = 1,2$ the composite
maps
 $B\pi_1(f\vert_{X_i}) \circ c_{X_i} \colon X_i \longrightarrow B\Gamma_i$
and the corresponding real homology classes
\begin{equation}
 \alpha_i := H_{\dim X_i}(B\pi_1(f\vert_{X_i})\circ c_{X_i})([X_i]) \in H_{\dim X_i}(B\Gamma_i;\R).
\end{equation}

Since $M$ is aspherical, and therefore the classifying map $c_M\colon M\longrightarrow B\pi_1(M)$ is homotopic to the identity, the commutative diagram in Figure \ref{f:KotschickLoeh} (cf.~\cite[Prop. 2.2]{KL}) implies that in degree $k$ homology we have
\begin{eqnarray*}
 0 \neq \alpha = (c_M\circ f)_*[X_1\times X_2]= (B\varphi)_*(\alpha_1 \times \alpha_2).
\end{eqnarray*}
This means that the $\alpha_i$ are not trivial and therefore
the $\Gamma_i$ are both infinite. 
Since $\pi_1(M)$ is not cyclic, we conclude that there exist elements $g_i\in\Gamma_i$ of infinite order such that 
\[
\Z\times\Z=\langle g_1\rangle\times\langle g_2\rangle \subset \pi_1(M).
\]
But this contradicts the fact that $\pi_1(M)$ is hyperbolic and finishes the proof.
\end{proof} 

Since negatively curved manifolds are aspherical and have hyperbolic fundamental group, we obtain the following consequence:

\begin{cor}\label{not representable}
If $M$ is a negatively curved manifold, then any non-trivial element $\alpha\in H_k(M;\R)$ is not realizable by products.
\end{cor}

\section{Proof of theorem \ref{t:main1}}\label{proof1}

In this section we prove Theorem \ref{t:main1} and give examples of manifolds that do not support transitive Anosov diffeomorphisms.

\subsection{Proof of Theorem \ref{t:main1}}

Let $M$ be a negatively curved manifold of dimension $m$ and $N$ be a rational homology sphere of dimension $n$.
Suppose $f\colon M\times N \longrightarrow M\times N$ is a codimension $k$ transitive Anosov diffeomorphism with orientable invariant distributions. By Theorem \ref{results} (b), there exists a (non-trivial) Ruelle-Sullivan class $\alpha\in H^k(M\times N;\R)$ and a positive $\lambda\neq1$ such that
\begin{equation}\label{eq.RS}
f^*(\alpha)=\lambda\cdot\alpha, \ \text{where} \ 0<k \leq \biggl[\frac{m+n}{2}\biggl].
\end{equation}
Moreover, Theorem \ref{results} (a) implies that $k\neq1$.

Since $N$ is a rational homology sphere, the K\"unneth theorem implies that the cohomology groups of $M\times N$ in degree $k$ are given by
\begin{equation}\label{eq.K}
H^k(M\times N;\R)\cong \left\{\begin{array}{ll}
        H^k(M;\R), & \text{for } k<n\\
        H^n(M;\R)\oplus H^n(N;\R), & \text{for } k=n\\
        H^k(M;\R)\oplus (H^{k-n}(M;\R)\otimes H^n(N;\R)), & \text{for } k>n.
        \end{array}\right.
\end{equation}

\subsection*{$M$ is a hyperbolic surface}
Before dealing with arbitrary dimensions, we examine the case when $M=\Sigma$ is a hyperbolic surface. Since $k\leq[(n+2)/2]$, we deduce that $k\leq2$. Thus $k=2$ and so $n\geq 2$. If $n>2$, then by (\ref{eq.K})
the Ruelle-Sullivan class $\alpha$ has the form 
\begin{equation}\label{surf1}
\alpha=\xi\cdot(\omega_\Sigma \times1)\in H^2(\Sigma;\R), \ \xi\in\R.
\end{equation}
Also, since $n>2$ and $\|\Sigma\|>0$, the effect of $f$ on $\omega_\Sigma$ is given by 
\begin{equation}\label{surf2}
f^*(\omega_\Sigma\times 1)=\pm(\omega_\Sigma\times1).
\end{equation}
By (\ref{eq.RS}), (\ref{surf1}) and (\ref{surf2}), we reach the absurd conclusion that $\lambda=\pm1$. Thus we may assume that $n=2$, i.e. $N=S^2$ is the 2-sphere. In that case, the Ruelle-Sullivan class $\alpha$ has the form
\begin{equation}\label{surf3}
\alpha=\xi_1\cdot(\omega_\Sigma\times 1)+\xi_2\cdot(1\times \omega_{S^2}), \ \xi_1,\xi_2\in\R.
\end{equation}
Since $\|\Sigma\|>0$ and $\|S^2\|=0$, we deduce that (\ref{surf2}) holds. Also, the effect of $f$ on $\omega_{S^2}$ is given by
\begin{equation}\label{surf4}
f^*(1\times\omega_{S^2})=\zeta\cdot(\omega_{\Sigma}\times 1)\pm(1\times\omega_{S^2}),
\end{equation}
where $\zeta\in\R$ and the coefficient $\pm1$ of $1\times\omega_{S^2}$ is because $\deg(f)=\pm1$. By (\ref{eq.RS}), (\ref{surf3}), (\ref{surf2}) and (\ref{surf4}), we deduce that $\xi_2=0$, because $\lambda\neq\pm1$. But then $\alpha=\xi_1\cdot(\omega_\Sigma\times 1)$, which is impossible as explained above.

We have now shown that the product of a hyperbolic surface with a rational homology sphere does not support transitive Anosov diffeomorphisms. 
From now on, we assume that $m\geq 3$. 


\subsection*{$M$ and $N$ have different dimensions}
First, let us assume that $m\neq n$. Since $M$ is negatively curved, we conclude that $M$ is not dominated by products~\cite{KL}. 
Thus, Thom's Realization Theorem \ref{t:Thom} implies that the effect of $f$ on the cohomological fundamental class of $M$ is 
\begin{equation}\label{eq1}
f^*(\omega_M\times1)=a\cdot(\omega_M\times 1), \ a\in\Z. 
\end{equation}
Also, since $H^l(N;\Q)=0$ for all $l\neq 0,n$, the effect of $f$ on the cohomological fundamental class of $N$ is
\begin{equation}\label{eq2}
f^*(1\times\omega_{N})=(\alpha^n_M\times 1)+c\cdot(1\times\omega_{N}), \ \alpha_M^n\in H^n(M;\R), \ c \in\Z.
\end{equation}
In particular, $ac=\deg(f)=\pm1$ which means that 
\begin{equation}\label{ac=1}
a,c\in\{\pm1\}. 
\end{equation}

We now claim 
that the image under $f^*$ of every cohomology class of $H^*(M;\R)$ remains in $H^*(M;\R)$.

\begin{prop}\label{effect of f on M}
$f^*(x_M^u\times 1)\in H^u(M;\R)$ for each $x_M^u\in H^u(M;\R)$.
\end{prop}
\begin{proof}
By (\ref{eq1}) and (\ref{ac=1}), we know that $f^*(\omega_M\times 1)=\pm (\omega_M\times 1)$, hence we may assume that $u<m$.
We have 
\begin{equation}\label{eq.effect}
f^*(x_M^{u}\times1)=(y_M^{u}\times1)+(y_M^{u-n}\times\omega_{N})\in H^{u}(M;\R)\oplus (H^{u-n}(M;\R)\otimes H^n(N;\R)).
\end{equation}

If $u<n$, then the claim holds trivially. If $u=n$, then (\ref{eq.effect}) becomes
\[
f^*(x_M^n\times1)=(y_M^n\times1)+\xi\cdot(1\times\omega_{N})\in H^n(M;\R)\oplus H^n(N;\R),
\]
for some $\xi\in\R$. By Poincar\'e duality there exists $x_M^{m-n}\in H^{m-n}(M;\R)$ such that $x_M^{n}\cup x_M^{m-n}=\omega_M$. We have
\[
f^*(x_M^{m-n}\times1)=(y_M^{m-n}\times1)+(y_M^{m-2n}\times\omega_{N})\in H^{m-n}(M;\R)\oplus (H^{m-2n}(M;\R)\otimes H^n(N;\R)).
\]
By (\ref{eq1}), (\ref{ac=1}) and the above last two equations we obtain
\begin{equation}\label{omegaM}
\begin{aligned}
\pm \ \omega_M\times 1 & =  f^*(\omega_M\times 1)\\
                                         & =  f^*(x_M^n\times 1)\cup f^*(x_M^{m-n}\times 1)\\
                                         & =  ((y_M^{n}\times1)+\xi\cdot(1\times\omega_N))\cup((y_M^{m-n}\times1)+(y_M^{m-2n}\times\omega_{N}).
\end{aligned}
 \end{equation}
This implies that $y_M^n\cup y_M^{m-n}=\pm \ \omega_M$ (in particular, $y_M^n$ and $y_M^{m-n}$ are not trivial) and $\xi=0$. This proves the proposition for $u=n$. 

Finally, let us assume that $u>n$. 
In this case $y_M^{u-n}=0$ in equation (\ref{eq.effect}), otherwise $x_M^u$ would be realizable by products by Thom's Realization Theorem \ref{t:Thom}, which is impossible by Corollary \ref{not representable} because $M$ is negatively curved. This finishes the proof. 
\end{proof}

We split the proof into three cases, according to the possible values of $k$.

\subsection*{Case I: $k<n$} 
By (\ref{eq.K}) we have that $\alpha\in H^k(M;\R)$. 
 But this is impossible because $M$ is negatively curved of dimension $m>2$, which means that $\Out(\pi_1(M))$ is finite and thus (\ref{eq.RS}) cannot hold. Indeed, recall that, by (\ref{eq1}) and (\ref{ac=1}), $f$ induces a self-map of $M$ of degree $\pm1$, given by
 \[
 M\stackrel{\iota_M}\hookrightarrow M\times N\stackrel{f}\longrightarrow M\times N \stackrel{p_M}\longrightarrow M,
 \]
where $\iota_M$ and $p_M$ denote inclusion and projection respectively; see~\cite{Neo} for further discussion. In particular, $(p_M\circ f\circ \iota_M)_*\colon \pi_1(M)\longrightarrow\pi_1(M)$ is surjective. Now since $M$ is negatively curved, $\pi_1(M)$ is hyperbolic and thus Hopfian, i.e. every surjective endomorphism of $\pi_1(M)$ is an isomorphism, and so $(p_M\circ f\circ\iota_M)_*$ is an automorphism. But $\Out(\pi_1(M))$ is finite and, by (\ref{eq.RS}), $p_M\circ f\circ\iota_M$ scales $\alpha$ by $\lambda$ in cohomology of degree $k$. This is a contradiction; see Section \ref{outer} for details.
 
\subsection*{Case II: $k=n$}
By (\ref{eq.K}) we have that 
\[
\alpha=(x_M^n\times1)+\nu\cdot(1\times\omega_{N})\in H^n(M;\R)\oplus H^n(N;\R),
\] 
for some $\nu\in\R$. We observe that $\alpha\notin H^n(M;\R)$ and $\alpha\notin H^n(N;\R)$. Indeed, first we see that $\alpha\notin H^n(M;\R)$ for the same reason as in Case I (because $M$ is negatively curved and $m>2$). Second, $\alpha\notin H^n(N;\R)$, because otherwise (\ref{eq.RS}), (\ref{eq2}) and (\ref{ac=1}) would imply 
$
\lambda\cdot(1\times\omega_{N})=\pm(1\times\omega_{N}),
$
and so $\lambda=\pm1$ which is impossible. 

Thus, using again (\ref{eq.RS}), (\ref{eq2}) and (\ref{ac=1}), we obtain
\begin{equation}\label{eq.k=n}
\lambda\cdot(x_M^n\times 1)+\lambda\nu\cdot(1\times\omega_{N})=f^*(x_M^n\times 1)+\nu\cdot(\alpha^n_M\times 1)\pm\nu\cdot(1\times\omega_{N}).
\end{equation}
By Proposition \ref{effect of f on M}, equation (\ref{eq.k=n}) becomes
\begin{equation}\label{eq.caseIIfinish}
\lambda\cdot(x_M^n\times 1)+\lambda\nu\cdot(1\times\omega_{N})=(y_M^n\times 1)+\nu\cdot(\alpha^n_M\times 1)\pm\nu\cdot(1\times\omega_{N}),
\end{equation}
which implies that $\lambda=\pm1$. This contradiction finishes the proof for $k=n$.

\subsection*{Case III: $k>n$}
By (\ref{eq.K}) we have that 
\[
\alpha=(x_M^k\times1)+(x_M^{k-n}\times\omega_{N})\in H^k(M;\R)\oplus (H^{k-n}(M;\R)\otimes H^n(N;\R)).
\]
As in Case I, we have $\alpha\notin H^k(M;\R)$ because $M$ is negatively curved and $m>2$. Next, we will show that $\alpha\notin H^{k-n}(M;\R)\otimes H^n(N;\R)$. Suppose the contrary, i.e. that $\alpha=x_M^{k-n}\times\omega_{N}\in H^{k-n}(M;\R)\otimes H^n(N;\R)$. By (\ref{eq.RS}) we have 
\begin{equation}\label{neweq.k>n}
\lambda\cdot(x_M^{k-n}\times\omega_{N})=f^*(x_M^{k-n}\times\omega_N).
\end{equation}
Thus, by Theorem \ref{results} (b), there exists a Ruelle-Sullivan class $x_M^{m-k+n}\in H^{m-k+n}(M;\R)$ such that 
\[
f^*(x_M^{m-k+n}\times 1)=\pm\frac{1}{\lambda}\cdot(x_M^{m-k+n}\times 1), \ \frac{1}{\lambda}\neq1.
\]
However, the latter equation is impossible because $M$ is negatively curved and $m>2$ (cf. Section \ref{outer}).
This proves that $\alpha\notin H^{k-n}(M;\R)\otimes H^n(N;\R)$.


Thus (\ref{eq.RS}) has the form 
\begin{equation}\label{RS:k>n}
\lambda\cdot(x_M^k\times1)+\lambda\cdot(x_M^{k-n}\times\omega_{N})=f^*(x_M^k\times1)+f^*(x_M^{k-n}\times\omega_{N}).
\end{equation}
By Proposition \ref{effect of f on M}, 
(\ref{eq2}) and (\ref{ac=1}), equation (\ref{RS:k>n}) becomes
 \begin{equation*}\label{eq.final}
 \begin{aligned}
\lambda\cdot(x_M^k\times 1)+\lambda\cdot(x_M^{k-n}\times\omega_{N}) & =  f^*(x_M^k\times 1) + f^*(x_M^{k-n}\times 1)\cup f^*(1\times\omega_N)\\
                                                                                                                     & =  (y_M^k\times 1) + ((y_M^{k-n}\cup\alpha_M^n)\times 1) \pm  (y_M^{k-n}\times\omega_{N}).
\end{aligned}
\end{equation*}
We deduce that $y_M^{k-n}=\pm\lambda\cdot x_M^{k-n}$, and so 
\begin{equation}\label{finallambda}
f^*(x_M^{k-n}\times 1)=\pm\lambda\cdot (x_M^{k-n}\times 1).
\end{equation}
As before, the last conclusion is impossible  
because $M$ is negatively curved and $m>2$. This completes the proof for the case $k>n$.

We have now finished the proof of Theorem \ref{t:main1} for $m\neq n$. 

\subsection*{$M$ and $N$ have the same dimension}

Let us now assume that $m=n$. Then $k\leq n$, because $k\leq[(m+n)/2]$. 

If $k<n$, then the result follows as in Case I above, so we can assume that $k=n$. Then the Ruelle-Sullivan class $\alpha$ has the form
\begin{equation}\label{eq.RS:k=n=m}
\alpha=\xi_1\cdot(\omega_M\times 1)+\xi_2\cdot(1\times\omega_N)\in H^m(M;\R)\oplus H^n(N;\R).
\end{equation}
The effect of $f$ on the cohomological fundamental classes of $M$ and $N$ respectively is
\begin{equation}\label{eqM:m=n}
f^*(\omega_M\times1)=a\cdot(\omega_M\times1)+a'\cdot(1\times\omega_N) 
\end{equation}
and
\begin{equation}\label{eqN:m=n}
f^*(1\times\omega_N)=c'\cdot(\omega_M\times1)+c\cdot(1\times\omega_N).
\end{equation}
Possibly after replacing $f$ with $f^2$, we may assume that $\deg(f)=1$, and so (\ref{eqM:m=n}) and (\ref{eqN:m=n}) give
\begin{equation}\label{sign}
ac+(-1)^m a'c'=1.
\end{equation}

If $ac\neq 0$, then $a'c'=0$ because $\|M\|>0$. In particular, $a=c=1$ or $a=c=-1$. By (\ref{eq.RS:k=n=m}), (\ref{eqM:m=n}) and (\ref{eqN:m=n}), equation (\ref{eq.RS}) becomes
 \begin{equation*}
 \begin{aligned}
\lambda\xi_1\cdot(\omega_M\times 1)+\lambda\xi_2\cdot(1\times\omega_{N})  & = 
\  \ \xi_1\cdot f^*(\omega_M\times 1) + \xi_2\cdot f^*(1\times\omega_N) \\
                                                                                                                     & = 
                                                                                                                     \left\{\begin{array}{ll}
        \xi_1a\cdot(\omega_M\times1)+(\xi_1a'+\xi_2c)\cdot(1\times\omega_N), & 
        a'\neq0\\
        \xi_1a\cdot(\omega_M\times1)+\xi_2c\cdot(1\times\omega_N), & 
        a',c'=0\\        
        (\xi_1a+\xi_2c')\cdot(\omega_M\times1)+\xi_2c\cdot(1\times\omega_N), & 
        c'\neq0.
                \end{array}\right.
\end{aligned}                
\end{equation*}
All three cases yield the absurd conclusion that $\lambda=\pm 1$.

If $a'c'\neq 0$, then $N$ dominates $M$ and so $\|N\|>0$. Similarly as above we deduce that $ac=0$. If $a=c=0$, then (\ref{eq.RS}), (\ref{eq.RS:k=n=m}), (\ref{eqM:m=n}) and (\ref{eqN:m=n}) imply that 
\[
\lambda\xi_1=\xi_2c' \ \text{and} \ \lambda\xi_2=\xi_1a'.
\]
Thus $\lambda^2=\pm1$ by (\ref{sign}). This is impossible, and so exactly one of $a$ and $c$ is zero. If $m$ is even, then $a'c'=1$ by (\ref{sign}). When $a\neq 0$, then $a=\pm 1$, and by (\ref{eqM:m=n}) and (\ref{eqN:m=n}) we obtain
\[
(f^2)^*(\omega_M\times 1)=2\cdot(\omega_M\times1)\pm(1\times\omega_N).
\]
This implies that $M$ admits a self-map of degree two, which contradicts the fact that $\|M\|>0$.
When $c\neq 0$, then similarly we obtain a self-map of $N$ of degree two, which is a contradiction because $\|N\|>0$. 
Finally, if $m$ is odd, then (\ref{sign}) implies that $a'c'=-1$. 
Since exactly one of $a$ and $c$ is zero, (\ref{eq.RS}), (\ref{eq.RS:k=n=m}), (\ref{eqM:m=n}) and (\ref{eqN:m=n}) imply that
\[
\lambda+\frac{1}{\lambda}=\pm1,
\]
which is impossible.
This finishes the proof of Theorem \ref{t:main1} for $m=n$.

The proof of Theorem \ref{t:main1} is now complete.

\subsection{Examples}\label{examples}

Using Theorem \ref{t:main1} we can construct new classes of manifolds which do not support transitive Anosov diffeomorphisms. 

As a first application, we obtain examples of manifolds with non-trivial higher homotopy groups that do support transitive Anosov diffeomorphisms. Corollary \ref{c:sphere} gives already such examples, when $N$ is a sphere. 
Every simply connected rational homology sphere of dimension at most four must be a sphere itself. Nevertheless, simply connected rational homology spheres that are not spheres exist in dimensions $\geq5$. Below we obtain examples of manifolds that do not support transitive Anosov diffeomorphisms using simply connected $5$-manifolds: 

\begin{ex}[Higher homotopy]\label{surfaceQHS}
Let $N$ be a simply connected $5$-manifold which does not contain any $S^3$-bundles over $S^2$ in its prime decomposition. Then $N$ is a rational homology sphere by the classification of simply connected 5-manifolds of Barden~\cite{Ba} and Smale~\cite{Sm}. Hence, Theorem \ref{t:main1} implies that $M\times N$ does not support transitive Anosov diffeomorphisms for any negatively curved manifold $M$. In particular, for any hyperbolic surface $\Sigma$, the product $\Sigma\times N$ does not support transitive Anosov diffeomorphisms.
\end{ex}

Concerning manifolds with trivial homotopy groups $\pi_k$ for $k\geq 2$ (i.e. aspherical manifolds), 
Gogolev-Lafont~\cite{GL} proved that the product of an infranilmanifold with finitely many negatively curved manifolds which have dimensions at least three does not support Anosov diffeomorphisms (the product of such negatively curved manifolds satisfies the properties of $M$ in Theorem \ref{resultsGHGL} (b)). Especially, as mentioned in Section \ref{outer}, the product of two negatively curved manifolds of dimensions at least three does not support Anosov diffeomorphisms. 
Under the transitivity assumption, Theorem \ref{t:main1} provides further examples of products of aspherical manifolds that do not support Anosov diffeomorphisms, including products of hyperbolic surfaces with certain higher dimensional negatively curved manifolds:

\begin{ex}[Aspherical products]
\label{surfaceQHS}
Let $N$ be a negatively curved manifold of dimension at least three which is a rational homology sphere. For instance, certain surgeries on hyperbolic knots in $S^3$ give infinitely many examples of hyperbolic integral homology spheres in dimension three~\cite{Th1,Th2}. In dimension four, examples of aspherical rational (but not integral) homology spheres were given in~\cite{Luo} and of aspherical integral homology spheres in~\cite{RT}. Now, Theorem \ref{t:main1} implies that for any negatively curved manifold $M$, the product $M\times N$ does not support transitive Anosov diffeomorphisms. In particular, if $M=\Sigma$ is a hyperbolic surface, then the product $\Sigma\times N$ does not support transitive Anosov diffeomorphisms.
\end{ex}

\section{Proof of theorem \ref{t:main2}}\label{proof2}

We now prove Theorem \ref{t:main2}, which extends Theorem \ref{t:main1} (and the examples of Section \ref{examples}) under additional assumptions on the dimensions of the negatively curved factors and the vanishing of the simplicial volume of the rational homology sphere. Some arguments and formulas will be taken from the proof of Theorem \ref{t:main1} and will not be repeated.

\medskip

Let $M_1,...,M_s$ be negatively curved manifolds of dimensions $m_1,...,m_s$ respectively, where $m_i\geq3$ for all $i=1,...,s$, and $N$ be a rational homology sphere of dimension $n$ such that $\|N\|=0$. 
For short, we set 
\[
M:=M_1\times\cdots\times M_s \  \text{and} \ m:=m_1+\cdots+m_s.
\]

Suppose $f\colon M\times N \longrightarrow M\times N$ is a codimension $k$ transitive Anosov diffeomorphism with orientable invariant distributions. 
As in the proof of Theorem \ref{t:main1}, let $\alpha\in H^k(M\times N;\R)$ be the Ruelle-Sullivan class that satisfies (\ref{eq.RS}) for some 
positive $\lambda\neq1$. Furthermore, we know by Theorem \ref{results} (a) that $k\neq1$.

Since $N$ is a rational homology sphere, the K\"unneth theorem implies that the real cohomology groups of $M\times N$ are given by (\ref{eq.K}).
Now, since $\|M\|\geq\|M_1\|\cdots\|M_s\|>0$ (cf.~\cite{G}) and $\|N\|=0$, we deduce that $f^*(\omega_M\times 1)$ is given as in (\ref{eq1}). Also, 
$f^*(1\times\omega_N)$ is given as in (\ref{eq2}). Note, finally, that the coefficients $a$ and $c$ that appear in equations (\ref{eq1}) and (\ref{eq2}) respectively must be $\pm 1$ (i.e. (\ref{ac=1}) holds), because $\deg(f)=\pm1$.

\medskip

We will now show that Proposition \ref{effect of f on M} holds in this situation as well. Although not every homology class of degree greater than one of $M=M_1\times\cdots\times M_s$ 
has non-vanishing simplicial semi-norm and is not realizable by products, we will use the fact that every homology class of degree greater than one of each factor $M_i$ still satisfies those two properties, and the additional assumption that the rational homology sphere $N$ has vanishing simplicial volume.

\begin{prop}\label{p:cohomologymap}
$f^*(x_M^u\times 1)\in H^u(M;\R)$ for each $x_M^u\in H^u(M;\R)$.
\end{prop}
\begin{proof}
The case $u=m$ is explained above 
and so we may assume that $u<m$. We have
\begin{equation}\label{eqclaim}
f^*(x_M^u\times 1)=(y_M^u\times 1)+(y_M^{u-n}\times\omega_N)\in H^u(M;\R)\oplus (H^{u-n}(M;\R)\otimes H^n(N;\R)),
\end{equation}
where
\[
x_M^u\in \mathop{\oplus}_{i=1}^s H^u(M_i;\R)\oplus\biggl(\mathop{\oplus}_{\substack{u_1+\cdots+u_s=u, \\ 0\leq u_j<u}}(H^{u_1}(M_1;\R)\otimes\cdots\otimes H^{u_s}(M_s;\R))\biggl),
\]
that is,
\begin{equation}\label{x^u expression}
x_M^u=\mathop{\sum}_{i=1}^s (1\times\cdots \times1\times x_{M_i}^u\times 1\times\cdots\times 1) \ +\mathop{\sum}_{\substack{u_1+\cdots+u_s=u, \\ 0\leq u_j<u}} (x_{M_1}^{u_1}\times\cdots\times x_{M_s}^{u_s}),
\end{equation}
for some $x_{M_i}^{*}\in H^{*}(M_i;\R)$. 

The proposition holds trivially if $u<n$, thus we may assume that $u\geq n$. If $u=n$, then (\ref{eqclaim}) has the form
\[
f^*(x_M^n\times1)=(y_M^n\times1)+\xi\cdot(1\times\omega_{N})\in H^n(M;\R)\oplus H^n(N;\R),
\]
for some $\xi\in\R$, and the proof follows as in Proposition \ref{effect of f on M}; cf. computation (\ref{omegaM}). (Note that the additional product structure on $H^*(M;\R)$ has no effect in the proof of Proposition \ref{effect of f on M} when $u=n$.)

Finally, let $u>n$. 
For each $x_{M_i}^u\in H^u(M_i;\R)$, we have
\begin{equation}\label{x^u}
f^*(1\times\cdots\times1\times x_{M_i}^u\times1\times\cdots\times1)\in \mathop{\oplus}_{i=1}^s H^u(M_i;\R), 
\end{equation}
otherwise $x_{M_i}^u$ would be realizable by some product class in 
\[
\biggl(\mathop{\oplus}_{\substack{u_1+\cdots+u_s=u, \\ 0\leq u_j<u}}(H^{u_1}(M_1;\R)\otimes\cdots\otimes H^{u_s}(M_s;\R))\biggl)\oplus(H^{u-n}(M;\R)\otimes H^n(N;\R)).
\]
But the latter is impossible by Corollary \ref{not representable}.
Next, each class
\[
x_{M_1}^{u_1}\times\cdots\times x_{M_s}^{u_s}\in H^{u_1}(M_1;\R)\otimes\cdots\otimes H^{u_s}(M_s;\R)
\]
is written as cup product
\[
\mathop{\cup}_{j=1}^s (1\times\cdots\times 1\times x_{M_j}^{u_j}\times 1\times\cdots\times 1).
\]
If $u_j>1$, then each $x_{M_j}^{u_j}$ has non-vanishing simplicial semi-norm and is not realizable by products. 
Thus the effect of $f$ in cohomology is given by
\begin{equation}\label{x^u_i}
f^*(1\times\cdots\times 1\times x_{M_j}^{u_j}\times 1\times\cdots\times 1)=\sum_{l=1}^s (1\times\cdots\times 1\times y_{M_l}^{u_l}\times 1\times\cdots\times 1), 
\end{equation}
for some $y_{M_l}^{u_l}\in H^{u_l}(M_l;\R)$.
Note that (\ref{x^u_i}) holds when $u_j=1$ as well. For if
 $u_j=1$, then the only case where $y_M^{u_j-n}\times\omega_N$ might not be trivial is when $n=u_j=1$. In that case,
 \[
 f^*(1\times\cdots\times 1\times x_{M_j}^{1}\times 1\times\cdots\times 1)=\sum_{l=1}^s (1\times\cdots\times 1\times y_{M_l}^{1}\times 1\times\cdots\times 1) +\xi\cdot(1\times\omega_{S^1}), 
  \]
for some $\xi\in\R$. Let $x_{M_j}^{m-1}=\omega_{M_1}\times\cdots\times x_{M_j}^{m_j-1}\times\cdots\times\omega_{M_s}$ such that $x_{M_j}^1\cup x_{M_j}^{m-1}=\omega_M$. Then
 \[
 f^*(x_{M_j}^{m-1}\times 1)=(y_M^{m-1}\times 1) + (y_M^{m-2}\times\omega_{S^1}), 
  \] 
  and the same computation as in (\ref{omegaM}) shows that $\xi=0$. Thus (\ref{x^u_i}) holds when $u_j=1$.
 Therefore, for each class $x_{M_1}^{u_1}\times\cdots\times x_{M_s}^{u_s}$ we obtain
 \begin{equation}\label{x^u_i product}
 f^*(x_{M_1}^{u_1}\times\cdots\times x_{M_s}^{u_s})\in H^u(M;\R). 
    \end{equation}

By (\ref{x^u expression}), (\ref{x^u}) and (\ref{x^u_i product}), we deduce that $y_M^{u-n}=0$ in (\ref{eqclaim}) and this finishes the proof.
\end{proof}

As in the proof of Theorem \ref{t:main1}, we now split the proof into three cases, according to the possible values of the codimension $k$.

\subsection*{Case I: $k<n$} 
By (\ref{eq.K}) we have that $\alpha\in H^k(M;\R)$. But this is impossible because $\Out(\pi_1(M))$ is finite, since $m_i\geq 3$ for all $i=1,...,s$; see Section \ref{outer}. 

\subsection*{Case II: $k=n$} 
By (\ref{eq.K}) we have that 
\[
\alpha=(x_M^n\times1)+\nu\cdot(1\times\omega_{N})\in H^n(M;\R)\oplus H^n(N;\R)
\]
and equation (\ref{eq.k=n}) holds. (Note that $\alpha\notin H^n(M;\R)$ as in Case I above. Also, $\alpha\notin H^n(N;\R)$, because otherwise (\ref{eq.RS}), (\ref{eq2}) and (\ref{ac=1}) would imply $\lambda=\pm1$ which is impossible.) 
By Proposition \ref{p:cohomologymap}, we have $f^*(x_M^n\times 1)\in H^n(M;\R)$. 
Thus equation (\ref{eq.k=n}) takes the form of (\ref{eq.caseIIfinish}), which leads to the absurd conclusion that $\lambda=\pm 1$.

\subsection*{Case III: $k>n$} 
By (\ref{eq.K}) we have that 
\[
\alpha=(x_M^k\times1)+(x_M^{k-n}\times\omega_{N})\in H^k(M;\R)\oplus (H^{k-n}(M;\R)\otimes H^n(N;\R)).
\]
As in Case I, we have $\alpha\notin H^k(M;\R)$. Also, $\alpha\notin H^{k-n}(M;\R)\otimes H^n(N;\R)$. Indeed, suppose the contrary, i.e. $\alpha=x_M^{k-n}\times\omega_{N}\in H^{k-n}(M;\R)\otimes H^n(N;\R)$. Then (\ref{eq.RS}) takes the form of (\ref{neweq.k>n}). 
By Theorem \ref{results} (b), there is a Ruelle-Sullivan class $x_M^{m-k+n}\in H^{m-k+n}(M;\R)$  
such that
\[
f^*(x_M^{m-k+n}\times 1)=\pm\frac{1}{\lambda}\cdot(x_M^{m-k+n}\times 1),
\]
which is in contradiction with the fact that $\Out(\pi_1(M))$ is finite; see Section \ref{outer}. This proves that $\alpha\notin H^{k-n}(M;\R)\otimes H^n(N;\R)$.


Thus we conclude that equation (\ref{RS:k>n}) holds. Now, Proposition \ref{p:cohomologymap} together with (\ref{eq2}) and (\ref{ac=1}) imply that equation (\ref{RS:k>n}) becomes
 \begin{eqnarray*}
\lambda\cdot(x_M^k\times 1)+\lambda\cdot(x_M^{k-n}\times\omega_{N}) & = & f^*(x_M^k\times 1) + f^*(x_M^{k-n}\times 1)\cup f^*(1\times\omega_N)\\
                                                                                                                     & = & (y_M^k\times 1) + ((y_M^{k-n}\cup\alpha_M^n)\times 1) \pm  (y_M^{k-n}\times\omega_{N}).
\end{eqnarray*}
This means that $y_M^{k-n} = \pm\lambda \cdot x_M^{k-n}$, and so $f^*(x_M^{k-n}\times 1) = \pm\lambda\cdot (x_M^{k-n}\times 1)$, which is impossible again because $\Out(\pi_1(M))$ is finite.

This finishes the proof of Theorem \ref{t:main2}.

\bibliographystyle{amsplain}

\begin{thebibliography}{123}

\bibitem{Ba}
D. Barden. {\em Simply connected five-manifolds}, Ann. of Math. {\bf 82}, No. 3 (1965), 365--385.

\bibitem{BF}
M. Bestvina and M. Feighn, {\em Stable actions of groups on real trees}, Invent. Math. {\bf 121} (2), (1995), 287--321.

\bibitem{Bow}
B. Bowditch, {\em Cut points and canonical splittings of hyperbolic groups}, Acta Math. {\bf 180} (2), (1998) 145--186.

\bibitem{Eilenberg:Steenrod}
S. Eilenberg, {\em On the problems of topology}, Ann. of Math. {\bf 50} (1949), 247--260.

\bibitem{Fr}
J. Franks, {\sl Anosov Diffeomorphisms. Global Analysis}, Proceedings of Symposia in Pure Mathematics, pp. 61--93, AMS, Providence, R.I 1970.

\bibitem{GH}
A. Gogolev and F. Rodriguez Hertz, {\em Manifolds with higher homotopy which do not support Anosov diffeomorphisms}, Bull. Lond. Math Soc. {\bf 46} no. 2, (2014) 349--366.

\bibitem{GL}
A. Gogolev and J.-F. Lafont, {\em Aspherical products which do not support Anosov diffeomorphisms}, Ann. Henri Poincar\'e {\bf 17} (2016), 3005--3026.

\bibitem{G}
M. Gromov, {\em Volume and bounded cohomology}, Inst. Hautes \'Etudes Sci. Publ. Math. {\bf 56} (1982), 5--99.

\bibitem{IY}
H. Inoue and K. Yano, {\em The Gromov invariant of negatively curved manifolds}, Topology {\bf 21} no. 1 (1981), 83--89.

\bibitem{KL}
D. Kotschick and C. L\"oh, {\em Fundamental classes not representable by products} J. Lond. Math. Soc. (2) {\bf 79} no. 3 (2009), 545--561.

\bibitem{KN}
D.~Kotschick and C.~Neofytidis, {\em On three-manifolds dominated by circle bundles}, Math. Z. {\bf 274} (2013), 21--32.

\bibitem{Luo}
F.~Luo, {\em The existence of $K(\pi_1,1)$ $4$-manifolds which are rational homology $4$-spheres}, Proc. Amer. Math. Soc. {\bf 104} (1988), 1315--1321.

\bibitem{Neo}
C. Neofytidis, {\em Degrees of self-maps of products}, Int. Math. Res. Not. IMRN {\bf 22} (2017), 6977--6989.

\bibitem{N}
S. E. Newhouse, {\em On codimension one Anosov diffeomorphisms}, Amer. J. Math. {\bf 92}, no. 3 (1970), 761--770.

\bibitem{Paulin}
F. Paulin, {\em Outer automorphisms of hyperbolic groups and small actions on R-trees}, Arboreal Group Theory, pp. 331--343. MSRI Publ. 19, Springer, New York 1991.

\bibitem{RT}
J. G. Ratcliffe and S. T. Tschantz, {\em Some examples of aspherical $4$-manifolds that are homology $4$-spheres}, Topology {\bf 44} (2005), 341--350.

\bibitem{RS}
D. Ruelle and D. Sullivan, {\em Currents, flows and diffeomorphisms}, Topology {\bf 14} (1975), 319--327.

\bibitem{Sh}
K. Shiraiwa, {\em Manifolds which do not admit Anosov diffeomorphisms}, Nagoya Math. J. {\bf 49} (1973), 111--115.

\bibitem{Sm}
S.~Smale, {\em On the structure of 5-manifolds}, Ann. of Math. (2) {\bf 75} (1962), 38--46.

\bibitem{Smale}
S.~Smale, {\em Differentiable dynamical systems}, Bull. Am. Math. Soc. {\bf 73} (1967), 747--817.

\bibitem{Thom}
R. Thom, {\em Quelques propri\'et\'es globales des vari\'et\'es diff\'erentiables}, Comment. Math. Helv. {\bf 28} (1954), 17--86.

\bibitem{Th1}
W.~Thurston, {\sl Three-Dimensional Geometry and Topology}, Princeton University Press, 1997

\bibitem{Th2}
W.~Thurston. {\em Three-dimensional manifolds, Kleinian groups and hyperbolic geometry}, Bull. Amer. Math. Soc. {\bf 6} (3), (1982), 357--381.

\bibitem{Yano}
K. Yano, {\em There are no transitive Anosov diffeomorphisms on negatively curved manifolds}, Proc. Japan Acad. Ser. A Math. Sci. {\bf 59} (9) (1983), 445.

\end{thebibliography}

\end{document}